\documentclass[a4paper, twoside, notitlepage, 12pt]{amsart} 
\usepackage{amsmath,amsfonts,amssymb,amsthm,mathrsfs} 
\usepackage[matrix,arrow]{xy}

\usepackage{verbatim}
\usepackage[latin1]{inputenc}   
\usepackage{backref}

\title[Weil Restriction and the  Quot scheme]{Weil Restriction and the Quot scheme}
\author{Roy Mikael Skjelnes}
\address{Department of Mathematics \\ KTH \\ Stockholm \\ Sweden}
\email{skjelnes@kth.se}

\subjclass[2010]{14A20, 14C05, 14D22}
\keywords{Weil restriction, Quot scheme, Fitting ideals}
\thanks{}

\date{280912}

\DeclareMathOperator{\id}{id}
\DeclareMathOperator{\im}{Im}
\DeclareMathOperator{\Spec}{Spec}

\newcommand{\ra}{\longrightarrow}

\newcommand{\Weilmodule}[3]{\mathscr{M}od_{{#1}\to {#2}}^{#3}}

\newcommand{\co}[1]{F({#1})}

\newcommand{\stackquot}{\mathfrak{Quot}}
\newcommand{\stackU}{\mathscr{U}}
\newcommand{\Coh}[1]{{\mathscr C}oh^n_{{#1}/S}}
\newcommand{\mono}[1]{\mathbf{{#1}}}
\newcommand{\corr}[1]{{#1}^{\dagger}}
\newcommand{\dual}[1]{{#1}^{\star}}

\newcommand{\Weil}[2]{\mathfrak{R}_{{#2}/{#1}}}
\newcommand{\NCWeil}[3]{\underline{\mathrm{Hom}}_{{#1}}({#2}, {#3})}

\newcommand{\Iso}[3]{\underline{\mathrm{Im}}^{#3}_{{#1}={#2}}}
\newcommand{\NN}{\mathbb{N}}

\newcommand{\Quot}[2]{\underline{\mathrm{Quot}}^n_{{#2}_{#1}/S}}

\newcommand{\eqbeg}{\begin{equation}}
\newcommand{\eqend}{\end{equation}}

\newcommand{\calA}{\mathscr{A}}
\newcommand{\calJ}{\mathscr{J}}
\newcommand{\calE}{\mathscr{E}}

\newcommand{\calO}{\mathscr{O}}
\newcommand{\calI}{\mathscr{I}}
\newcommand{\calM}{\mathscr{M}}

\newcommand{\calF}{\mathscr{F}}

\renewcommand{\hom}{\mathrm{Hom}}
\newcommand{\Snbimap}{\xymatrix@M=1pt{ {\Sgot_n\times U}
\ar@<.5ex>[r] \ar@<-.5ex>[r] & {U}}}
\newcommand{\bimap}{\xymatrix@M=1pt{ {R}
\ar@<.5ex>[r] \ar@<-.5ex>[r] & {X}}}
\newcommand{\pbbimap}{\xymatrix@M=1pt{ {R\times_{A}Y}
\ar@<.5ex>[r] \ar@<-.5ex>[r] & {X\times_AY}}}
\newcommand{\grpbimap}{\xymatrix@M=1pt{ {G\times X}
\ar@<.5ex>[r] \ar@<-.5ex>[r] & {X}}}
\newcommand{\pibimap}{\xymatrix@M=1pt{ {R}
\ar@<.5ex>[r]^-{\pi_1} \ar@<-.5ex>[r]_-{\pi_2} & {X}}}
\newcommand{\quotmap}{\xymatrix@M=1pt{ {R/G}
\ar@<.5ex>[r]^-{\pi_1} \ar@<-.5ex>[r]_-{\pi_2} & {U/G}}}

\newtheorem{summary}{Theorem}
\newtheorem{thm}[subsection]{Theorem}
\newtheorem{lemma}[subsection]{Lemma}
\newtheorem{cor}[subsection]{Corollary}
\newtheorem{prop}[subsection]{Proposition}

\theoremstyle{definition}
\newtheorem{defn}[subsection]{Definition}
\newtheorem{ex}[subsection]{Example}

\theoremstyle{remark}
\newtheorem{rem}[subsection]{Remark}

\numberwithin{equation}{subsection}

\begin{document}

\begin{abstract} We introduce a concept that we call module
  restriction, which generalises the classical Weil restriction. After
  having 
  established some fundamental properties, as existence and
  \'etaleness, we apply our
  results to show that the Quot functor $\Quot{X}{F}$ of
  Grothendieck is representable by an algebraic space, for any quasi-coherent sheaf $F_X$ on
  any separated algebraic space $X/S$.  
\end{abstract}

\maketitle

\section*{Introduction}

The main novelty in this article is the introduction of the {\em
  module restriction}, which is a generalisation of the classical  Weil restriction. Our main motivation for
introducing the module restriction is given by our application
to the Quot functor of Grothendieck.

If $F_X$ is a quasi-coherent sheaf on a scheme $X\ra S$, then the Quot
functor $\underline{\operatorname{Quot}}_{F_X/S}$ parametrises quotients of $F_X$
that are flat and with proper support over the base. For projective
schemes $X\ra S$ the Quot functor is represented by a scheme given as
a disjoint union of projective schemes \cite{FGAIV}. When $X\ra S$ is
locally of finite type and separated, Artin showed that the Quot
functor is representable by an algebraic space (\cite{Artinformalmoduli} and erratum in
\cite{ArtinDeformation}). 

When the fixed sheaf $F_X={\calO}_X$ is the structure sheaf of $X$ the
Quot functor is referred to as the Hilbert functor $\underline{\operatorname{Hilb}}_{X/S}$.

Grothendieck who both introduced the Quot functor and showed
representability for projective $X\ra S$, also pointed out the
connection between the Hilbert scheme and the Weil restriction
\cite[4. Variantes]{FGAIV}. If $f\colon Y \ra X$ is a morphism with
$X$ separated over the base, there is an open subset $\Omega_{Y\to
  X}$ of
$\underline{\operatorname{Hilb}}_{Y/S}$ from where the push-forward
map $f_*$ is defined. The
fibres of $f_* \colon \Omega_{Y\to X}\ra  \underline{\operatorname {Hilb}}_{X/S}$  are identified with the Weil restrictions.

However, even though the Weil restriction appears naturally in
connection with Hilbert schemes, there does not seem to exist any description of
the more general situation with the Quot scheme replacing the Hilbert
scheme. The purpose of this article is to give such a description with
the Quot functor $\Quot{X}{F}$ parametrising quotients of $F_X$
that are flat, with {\em finite} support and of relative rank $n$, over the base. In order to do
so we will need to introduce a generalisation of the Weil restriction.

Generalisations of the Weil restriction exist (\cite{Olsson} and \cite{SGA4}), but those
generalisations leap off in different directions than what is needed
for the present discussion. The generalisation we undertake here is in the
direction from ideals to modules. 

We fix a homomorphism of
$A$-algebras $B \ra R$, and a $B$-module $M$. The
module restriction
$\Weilmodule{B}{R}{M}$ parametrises, as a functor from $A$-algebras to
sets, $R$-module structures extending the
fixed $B$-module structure on $M$.

When $M$ is  finitely generated and projective as
an $A$-module, we 
show that the module restriction $\Weilmodule{B}{R}{M}$ is
representable by an $A$-algebra. We show representability by constructing the representing object in
the free algebra situation, and using Fitting ideals in the general
situation.  

Furthermore, when the $B$-module $M$ is a quotient
 of $B$, then we obtain that the module restriction  $\Weilmodule{B}{R}{M}$ coincides with  the Weil restriction.

These observations are summarised by the following result.

\begin{summary} Let $X\ra S$ be a separated morphism of schemes (or
  algebraic spaces) and let ${\mathscr C}oh_{X/S}^n$ denote the stack of
quasi-coherent sheaves on $X$ that are flat, with finite support and of
relative rank $n$ over the base $S$. For any affine morphism $f\colon
Y \ra X$ the push-forward map
$$
(\star) \quad f_* \colon {\mathscr C}oh_{Y/S}^n \ra {\mathscr C}oh_{X/S}^n
$$
is schematically representable. 
\end{summary}

The fibers of the push-forward map $(\star)$ are the module restrictions
parametrising sheaves on ${\calF}$ on $Y$ that are flat, finite, and
of relative rank $n$ over the base, such that the
push-forward $f_*\calF$ is isomorphic to a fixed $\calE$ on $X$. 

When the morphism $f\colon
Y \ra X$ is \'etale, then the the push-forward map $(\star )$
is not in general \'etale. The push-forward map $(\star)$ is only \'etale when
restricted to the open substack ${\stackU}_{Y\to X}$ consisting of
sheaves ${\calF} \in \Coh{Y}$ on $Y$, such that the induced map
of supports
$$ \mathrm{Supp}({\calF}) \ra \mathrm{Supp}(f_*{\calF})$$
is an isomorphism.

These requirements concerning the support of the sheaves, highlights differences between the Quot functor and the Hilbert functor. Fibers of $(\star)$ over $\calE=\calO_Z$, structure sheaves of closed subschemes $Z\subseteq X$, are \'etale when $f\colon X \ra Y$ is \'etale. 
 
Let $f_* \colon {\stackU}_{Y\to X} \ra \Coh{X}$ also denote the restriction of $(\star )$ to the open substack where the induced map of supports is an isomorphism. If we denote $Z=\mathrm{Supp}({\calE})$, the support  of a given element $\calE$ in $\Coh{X}$, then the fiber over $\calE$ is the Weil restriction of $Y\times_XZ \ra Z$. Thus, the Weil restriction appears naturally in the more general context with the Quot functors as well. Even though the support $Z=\mathrm{Supp}(\calE) \ra S$ is not necessarily flat, it turns out that in our situation the Weil restriction of $Y\times_XZ \ra Z$ still exists as a scheme.

Having established these technical results concerning the support, the representability of $\Quot{X}{F}$ follows easily. Let
$F_Y$ denote the pull-back of the quasi-coherent sheaf $F_X$ along
$f\colon Y\ra X$. There is a natural, forgetful, map $\Quot{X}{F} \ra
\Coh{X}$ whose pull-back along the push-forward map $(\star )$
restricted to ${\stackU}_{Y\to X}$, gives a representable, \'etale
covering
$$ \Omega_{Y\to X}^F \ra \Quot{X}{F}.$$
This \'etale cover specializes in the Hilbert functor situation, that is with $F_X=\calO_X$, to  the classical cover mentioned eariler, with the fibers being Weil restriction. We obtain the following result.

\begin{summary} Let $X\ra S$ be a separated map of algebraic spaces,
  and $F_X$ a quasi-coherent sheaf on $X$. Then the Quot functor
  $\Quot{X}{F}$ is representable by an algebraic space.
\end{summary}

In particular this generalises the result about the representability
of the Hilbert functor $\underline{\operatorname{Hilb}}^n_{X/S}$
  described in \cite{ES}. See also the generalisation to Hilbert
  stacks in \cite{RydhHilb}. The result also extends the earlier mentioned result of
  Artin in the sense that we do not assume the map $X\ra S$ to
  be of finite type, and there is no restriction on the base $S$.

\subsection*{Acknowledgements} Comments and corrections from Dan
Laksov and David
Rydh were important for the presentation of this
manuscript. Discussion with Runar Ile about non-commutative ring
theory were also helpful and clarifying.

\section{Fitting ideals}


We will in this first section point out some facts about
Fitting ideals that we will use later on.


\subsection{Conventions} A commutative ring $A$ is always a unital commutative ring. The
category of $A$-algebras, means the category of commutative
$A$-algebras. 

\begin{lemma}\label{fitting} Let $E$ be a projective $A$-module of rank
  $n$. Let $E\ra Q$ be a quotient module, and let $F_{n-1}(Q)\subseteq
  A$ denote the $(n-1)$'st Fitting ideal of $Q$. Then the $A$-module
  map $E\ra Q$ is an
  isomorphism if and only if $F_{n-1}(Q)=0$ is the zero ideal. In
  particular we have that a ring homomorphism $A \ra A'$ will
  factorise via $A/F_{n-1}(Q)$ if and only if  $E\bigotimes_AA '\ra
  Q\bigotimes_AA'$ is an isomorphism.
\end{lemma}

\begin{proof} The statement can be checked locally on $A$, hence we
  may assume that $E$ is free of finite rank $n$. The result then
  follows from the definition of the Fitting ideal.
\end{proof}

\subsection{Rank of projective modules} The rank of a projective module $E$ is constant on the
connected components of $\Spec(A)$. We will employ  the following
notation. Let $I\subseteq E$ be a submodule of a finitely generated
and projective module $E$. We let
$$\mathrm{Fitt}(I) : =F_{\mathrm{rk}E-1}(E/I) \subseteq A$$
denote the Fitting ideal we obtain by assigning on each connected component
of $\Spec(A)$ the Fitting ideal $F_{n-1}(E/I)$, where $n$ is the rank of
$E$ on that particular component.

\subsection{Closed conditions} Let $F$ be a co-variant functor from the
category (or a subcategory) of $A$-algebras  to sets. We
say that $F$ is a {\em closed condition} on $A$ if the functor $F$ is
representable by a quotient algebra of $A$.

\begin{prop} Let $\xi \colon R\ra E$ be an
  $A$-module homomorphism. Assume that  $E$ is finitely generated and projective as an
  $A$-module.
\begin{enumerate}
\item Let $I\subseteq R$ be a submodule. Then $\xi$
  factorising via the quotient map $R \ra R/I$ is a closed condition
  on $A$. 
\item Let $\xi'\colon R \ra E$ be another $A$-module homomorphism. Then
  $\xi $ being equal to $\xi'$ is a closed condition on $A$.
\end{enumerate}
\end{prop}

\begin{proof} In the first situation consider the Fitting ideal
  $\mathrm{Fitt}(I_1)$, of the quotient module of $E$ given by
  $I_1=\xi(I)$. In the second situation consider the Fitting ideal
  $\mathrm{Fitt}(I_2)$, of the quotient module of $E$ given by the $A$-submodule 
$$ I_2 =\{\xi(x)-\xi'(x) \mid x \in R\}.$$
It then follows from Lemma \ref{fitting} that assertions 1 and 2 
are represented by the quotient algebras
$A/\mathrm{Fitt}(I_1)$ and $A/\mathrm{Fitt}(I_2)$,  respectively.
\end{proof}

\subsection{}  With an $A$-algebra $E$, with
$E$ not necessarily commutative, we mean a unital ring homomorphism $c\colon A
\ra E$ from a commutative ring $A$, to an associative, unital ring
$E$, and where the image $c(A)$ is contained in the centre
of $E$ \cite{Bourbaki_Alg1_3}.

\begin{cor}\label{closed conditions} Let $\xi \colon R \ra E$ be an $A$-algebra homomorphism
  between two not necessarily commutative, $A$-algebras. Assume that
  $E$ is finitely generated and projective as an $A$-module.
\begin{enumerate}
\item Let $I\subseteq R$ be a two-sided ideal. Then $\xi$ factorising
  via the quotient map $R \ra R/I$ is a closed condition on $A$.
\item Let $\xi'\colon R \ra E$ be another $A$-algebra
  homomorphism. Then $\xi$ being equal to $\xi'$ is a closed condition
  on $A$.
\end{enumerate}
\end{cor}

\begin{proof} In both cases the question is whether a submodule of $E$
  is zero or not, and then the statement follows from the proposition.
\end{proof}
\begin{prop}\label{representing tensor product} Let $\xi_i \colon R_i \ra E$ be two $A$-algebra
  homomorphisms between not necessarily commutative
  $A$-algebras $(i=1,2)$. Assume that $E$ is finitely generated and projective
  as an $A$-module. Then the condition that $\xi_1$ commutes with
  $\xi_2$ is a closed condition on $A$. In particular, if $R_1$ and
  $R_2$ are commutative, then the two $A$-algebra homomorphisms $\xi_i
  \colon R_i \ra E$ factorising via $R_1\bigotimes_AR_2$ is a closed
  condition on $A$.
\end{prop}

\begin{proof} We consider the Fitting ideal
$\mathrm{Fitt}(I)$, where $I\subseteq E$ is the $A$-submodule
$$I =\{\xi_1(x)\xi_2(y)-\xi_2(y)\xi_1(x) \mid x\in R_1, y\in R_2\}.$$
The condition that $\xi_1$ commutes with $\xi_2$ is that the module
$I$ is the zero module. Hence, by Lemma \ref{fitting} we get that
$A/\mathrm{Fitt}(I)$ represents this condition. To prove the second
statement, consider the induced commutative diagram of $A$-algebras and
$A$-algebra homomorphisms
$$\xymatrix{ R_1\bigotimes_A A/\mathrm{Fitt}(I) \ar[r]^{\xi_1\otimes
    1} & E\bigotimes_AA/\mathrm{Fitt}(I) \\
A/\mathrm{Fitt}(I) \ar[u]\ar[r] & R_2\bigotimes_AA/\mathrm{Fitt}(I) \ar[u]^{\xi_2
  \otimes 1}}.$$
Since the images of $\xi_1\otimes 1$ and $\xi_2 \otimes 1$ commute,
over $A/\mathrm{Fitt}(I)$, we get an induced $A$-algebra homomorphism
$$\xymatrix{ \xi_1\otimes \xi_2 \colon R_1\bigotimes_AR_2
  \bigotimes_AA/\mathrm{Fitt}(I) \ar[r] &
  E\bigotimes_AA/\mathrm{Fitt}(I)},$$
sending $x_1\otimes x_2 \otimes a$ to $\xi_1(x_1)\xi_2(x_2)\otimes
a$. Conversely, let $A\ra A'$ be an $A$-algebra, and assume that the
two $A$-algebra homomorphisms $\xi_i \colon R_i \ra E\bigotimes_AA'$
factorises via $R_1\bigotimes_AR_2$. Let $\xi \colon R_1
\bigotimes_AR_2 \ra E\bigotimes_AA'$ denote the induced map. Then in
particular the image of $\xi$ is a commutative $A$-algebra. By the
usual properties of the tensor product we have that $\xi$ is the pair
$(\xi_1,\xi_2)$. Moreover, since the image of $\xi$ is a commutative
subring of $E\bigotimes_AA'$, we have that the images of $\xi_i \colon
R_1 \ra E\bigotimes_AA'$ commute. Then, by the above result, we get
that the homomorphism $A\ra A'$ factorises via $A/\mathrm{Fitt}(I)$.
\end{proof}

\subsection{Trace map} Let $E$ be an $A$-module. The dual module
$\mathrm{Hom}_A(E,A)$ we will denote by $\dual{E}$. The trace map is
the induced $A$-module homomorphism
   \begin{equation}\label{tracemap}
\xymatrix{\operatorname{Tr}\colon  E\bigotimes_A\dual{E}
  \ar[r]& A}.
\end{equation}

\begin{prop} Let $I\subseteq E$ be an inclusion of $A$-modules, where $E$ is finitely
  generated and projective. 
Then we have the identity of ideals
$\operatorname{Tr}(I\bigotimes_A\dual{E})=\mathrm{Fitt}(I)$ in $A$, where
  $\operatorname{Tr}$ denotes the trace map \ref{tracemap}.
\end{prop}

\begin{proof} Both the Fitting ideal and the trace map $\operatorname{Tr}$ commute with
  base change, and we may therefore assume that $E$ is free as an
  $A$-module.  Take a
  presentation of the $A$-module  $I$. That is we consider the map of free
  $A$-modules
\begin{equation}\label{freeAmodulemap} \xymatrix{\bigoplus_{\alpha
    \in\calA}Ae_{\alpha} \ar[r] & E}\end{equation}
determined by sending $e_{\alpha}$ to $f_{\alpha}$, where
$\{f_{\alpha}\}_{\alpha \in \calA}$ is a collection of generators of
the $A$-module $I$. The
cokernel of the map \ref{freeAmodulemap} is by definition $E/I$,
and consequently the $(n-1)$-minors of the map generate the Fitting
ideal $\mathrm{Fitt}(I)$. Let $e_1, \ldots, e_n$ be a basis for
$E$. Any $f \in E$  is then uniquely written as
$$f=\sum_{k=1}^nf_k^Ee_k,$$
with $f_1^E, \ldots, f_n^E$ in $A$. We obtain then that the
$(n-1)$-minors of the map \ref{freeAmodulemap} are
$\{(f_{\alpha})_1^E, \ldots, (f_{\alpha})^E_n\}_{\alpha \in
  \calA}$. We can now relate the Fitting ideal $\mathrm{Fitt}(I)$
to the other ideal $\operatorname{Tr}(I\bigotimes_A\dual{E})$. Let $\dual{e}_1, \ldots, \dual{e}_n$ denote
the dual basis for $\dual{E}$. For each $k=1, \ldots, n$ we have that
$$\operatorname{Tr}(f\otimes\dual{e}_k)=(\sum_{i=1}^nf_i^Ee_i)\otimes\dual{e}_k=f_k^E.$$
Thus $\{(f_{\alpha})_1^E ,\ldots, (f_{\alpha})^E_n\}_{\alpha \in
  \calA}$ also generate the $\operatorname{Tr}(I\bigotimes_A\dual{E})$, and we
have proven the equality of ideals. 
\end{proof}

\section{Parametrising algebra homomorphisms}

\subsection{Notation} If $V$ is an
$A$-module, we let $A[V]$ denote the
symmetric quotient algebra of the full tensor algebra
$T_A(V)=\bigoplus_{n\geq 0} V^{\otimes ^n}$.

\subsection{Preliminaries} Let $A\ra R$ and $A\ra E$ be two, not
necessarily commutative, $A$-algebras. We consider the functor
$\NCWeil{A}{R}{E}$, that to each commutative $A$-algebra $A'$,
assigns the set
$$\xymatrix{ \mathrm{Hom}_{A\text{-alg}}(R, E\bigotimes_AA')}.$$

\begin{rem}Since $E$ is an $A$-algebra, the tensor product
  $E\bigotimes_AA'$ exists and is an $A'$-algebra.
\end{rem}

\begin{prop}\label{mapsfromTV} Let $A$ be a commutative ring, and let
  $c\colon A \ra E$ be  an $A$-algebra, where
  $E$ is not necessarily commutative. Assume that $E$ is finitely
  generated and projective as an $A$-module. For any $A$-module $V$ we
  have that the $A$-algebra $A[V\bigotimes_A\dual{E}]$
  represents $\NCWeil{A}{R}{E}$, with $R=T_A(V)$.
\end{prop}
\begin{proof} An $A$-algebra homomorphism $u_{A'} \colon T_A(V) \ra E\bigotimes_AA'$ is
  determined by an $A'$-linear map $u_1\colon V\bigotimes_AA'\ra
  E\bigotimes_AA'$. Since $E$ is finitely generated and projective as
  an $A$-module, the $A'$-linear map $u_1$ is equivalent with an
  $A'$-linear map  $\varphi_1 \colon V\bigotimes_AA'
  \bigotimes_{A'}\dual{(E\bigotimes_AA')} \ra A'$. Moreover, the canonical map
$$\xymatrix{\mathrm{Hom}_A(E,A)\bigotimes_AA' \ar[r] &
  \mathrm{Hom}_{A'}(E\bigotimes_AA',A')}$$
is an isomorphism (\cite[4.3. Proposition 7]{Bourbaki_Alg1_3}).
   Therefore $\varphi_1$
corresponds  to an $A$-algebra homomorphism $\varphi \colon
A[V\bigotimes_A\dual{E}] \ra A'$. See e.g. \cite{Dieudonne}, or \cite{Bourbaki_Alg1_3}.
\end{proof}

\begin{cor}\label{AmapsfromR}  Let $A\ra R$ and
  $A \ra E$ be two, not necessarily commutative,  $A$-algebras. Assume that $E$ is finitely generated and projective as
  an $A$-module. Then $\NCWeil{A}{R}{E}$ is representable.
\end{cor}

\begin{proof} Write $R=T_A(V)/I$, for some two-sided ideal
  $I\subseteq T_A(V)$, for some $A$-module $V$. By Proposition
  \ref{mapsfromTV} the functor
  $\NCWeil{A}{T_A(V)}{E}$ is representable by
  $H=A[V\bigotimes_A\dual{E}]$. Let $\mu
  \colon T_A(V)\bigotimes_AH \ra E\bigotimes_AH$ denote the universal
  map. We are interested describing those $A$-algebra maps $H\ra A'$ that factorise via $R$.  The result now follows by applying Corollary \ref{closed
    conditions} (1), to the $H$-algebra $\mu$ (and where the objects in the category are $H$-algebras, and the morphisms are $H$-algebra homomorphisms that also are $A$-linear). \end{proof}

\subsection{Non-commutative Weil restrictions} Let $g\colon A \ra B $ and
$f\colon B \ra R$ be
homomorphisms of commutative rings. Let $c\colon A \ra E$ be
an $A$-algebra, where $E$ is not necessarily commutative, and assume
that we have an $A$-algebra homomorphism $\mu \colon B \ra E$.
Thus we fix the following data
\begin{equation}\label{thesetup} \xymatrix{ A \ar[r]^g \ar[rd]^{c} & B \ar[r]^f \ar[d]^{\mu} & R \\
  & E & }
\end{equation}
where $c=\mu \circ g$. We consider the functor $\NCWeil{B}{R}{E}$,
from the category of commutative $A$-algebras to sets,  that
for any $A$-algebra $A'$ assigns the set of $B$-algebra homomorphisms
$$ \xymatrix{\operatorname{Hom}_{B\text{-alg}}(R, E\bigotimes_AA').}$$
With a  $B$-algebra homomorphism  $\xi \colon R \ra E\bigotimes_AA'$
we mean an $A$-algebra homomorphism, that also is $B$-linear.

\begin{rem} It is important for applications that we have in mind that we do not
  assume that $\mu \colon B \ra E$ is  a $B$-algebra.
\end{rem}

\begin{rem} Let $A'$ be an $A$-algebra, and let $\xi \colon R\ra
E\bigotimes_AA'$ be a $B$-algebra homomorphism. The morphism $\xi$
will factorise as a $B\bigotimes_AA'$-algebra homomorphism $R\bigotimes_AA' \ra
E\bigotimes_AA'$, which is the identity on $A'$. Therefore the set set
of $A'$-valued points of $\NCWeil{B}{R}{E}$ is the set
$$\xymatrix{ \mathrm{Hom}_{B\bigotimes_AA'\text{-alg}}(R\bigotimes_AA',E\bigotimes_AA').}$$ 

\end{rem}


\begin{rem} When $E$ is commutative, and equal to $E=B$, then the
  functor $\NCWeil{B}{R}{B}$ is by definition the Weil restriction $\Weil{R}{B}$,
  see e.g. \cite{BLR}.
\end{rem}

\begin{prop} Let $g\colon A\ra B$ and $h \colon A \ra D$ be
  homomorphism of commutative rings, let $\mu \colon B \ra E$ be an
  $A$-algebra homomorphism, where $E$ is finitely
  generated and projective as an $A$-module. Then
  $\NCWeil{B}{B\bigotimes_AD}{E}$ is representable.
\end{prop}

\begin{proof} By Lemma \ref{AmapsfromR} the functor $\NCWeil{A}{D}{E}$
  is representable. Let $H$ be the representing object, and let
  $u\colon D\bigotimes_AH \ra E\bigotimes_AH$ be the universal
  map. Let $\mu\otimes 1 \colon B\bigotimes_AH \ra E\bigotimes_AH$ denote the induced map we get from the fixed
$A$-algebra homomorphism $\mu \colon B \ra  E$. By Proposition
\ref{representing tensor product} there is a quotient algebra $H \ra H/I$
representing the closed condition where  $\mu\otimes 1$ and $u$ commute. We have that the restrictions of the two maps to $H/I$  factorise as
$$\xymatrix{ \mu \otimes u \colon B\bigotimes_AD \bigotimes_AH/I \ra E\bigotimes_AH/I}.$$
It follows that $H/I$ represents $\NCWeil{B}{B\bigotimes_AD}{E}$.
 \end{proof}

\begin{cor}\label{NCWeil is rep} Let $A\ra B \ra R$ be homomorphism of commutative rings,
  and let $\mu \colon B \ra E$ be an $A$-algebra homomorphism, where
  the $A$-algebra $E$ is not necessarily commutative, but is finitely
  generated and projective as an $A$-module. Then $\NCWeil{B}{R}{E}$
  is representable.
\end{cor}

\begin{proof} Write $R$ as a quotient $T_A(V)\bigotimes_AB/I$ for some $A$-module
  $V$, and some two-sided ideal $I\subseteq T_A(V)\bigotimes_AB$. By
  the proposition we have that $\NCWeil{B}{B\bigotimes_AT_A(V)}{E}$ is
  representable. Let $H$ denote the representing object, and let
  $u\colon T_A(V)\bigotimes_AB\bigotimes_AH \ra E\bigotimes_AH$ denote
  the universal element. The result then follows by Corollary
  \ref{closed conditions} (1).
\end{proof}

\subsection{A note on Weil restrictions} Let $A\ra B$ be a
homomorphism of commutative rings, and assume that $B$ is finitely
generated and projective as an $A$-module. Let $V$ be an
$A$-module. By Proposition (\ref{mapsfromTV}) we have that
$A[V\bigotimes_A\dual{B}]$ represents $\NCWeil{A}{T_A(V)}{B}$. Let 
$$\xymatrix{u \colon T_A(V) \ra
  B\bigotimes_AA[V\bigotimes_A\dual{B}]}
$$
denote the universal map. We get by extension of scalars an induced
$B$-algebra homomorphism
\begin{equation}\label{universal u2} \xymatrix{u_B \colon B\bigotimes_A T_A(V) \ra
  B\bigotimes_AA[V\bigotimes_A\dual{B}]}.
\end{equation}
For any ideal $I\subseteq B\bigotimes_AT_A(V)$, we let $u_B(I)$ denote
the $A[V\bigotimes_A\dual{B}]$-module generated by the image of $I$.
\begin{cor}\label{Weil1stversion} Let $g\colon A\ra B$ be a
  homomorphism of commutative rings, where $B$ is
  finitely generated and projective as an $A$-module. Let $f\colon B
  \ra R$ be homomorphisms of rings. Write
  $R=B\bigotimes_AT_A(V)/I$ as a quotient of the full tensor algebra,
  where $V$ is some $A$-module. Then the Weil
  restriction $\Weil{B}{R}=\NCWeil{B}{R}{B}$ is representable by  the $A$-algebra  
$$\xymatrix{A[V\bigotimes_A\dual{B}]/\mathrm{Fitt}(u_B(I)) },$$
where $u_B$ is the universal map \ref{universal u2}.
\end{cor}
\begin{proof} By Proposition \ref{mapsfromTV} we have that the $A$-algebra
  $A[V\bigotimes_A\dual{B}]$ represents $\NCWeil{A}{T_A(V)}{B}$. Since
  $B$ is commutative, it follows that $A[V\bigotimes_A\dual{B}]$ also represents
  $\NCWeil{B}{T_A(V)}{B}$. Then, finally, the result follows from
  Corollary \ref{closed conditions} (1).
\end{proof}

\begin{rem} The defining properties of the full tensor algebra
  $T_A(V)$ as well as the symmetric quotient $A[V]$ are well-known,
  and can be found in e.g. \cite{Dieudonne} and \cite{Bourbaki_Alg1_3}. The situation with the
  Weil restriction as in Corollary \ref{Weil1stversion}, can be found in
  e.g. \cite[Theorem 7.4]{BLR}. 
\end{rem}


\begin{ex} We will in this example explicitly describe the
  correspondence between maps from $T_A(V)$ to $E$, and the
  representing object $C^E_V=A[V\bigotimes_A
  \dual{E}]$, given by Proposition \ref{mapsfromTV} and its corollary. Assume that the $A$-algebra $E$ is free as an
$A$-module with basis
$e_1, \ldots , e_n$, and let $V$ be a free $A$-module with basis
$\{t_i\}_{i\in \calI}$. For any monomial $f=t_{i_1}\otimes \cdots
\otimes t_{i_p}$ in $T_A(V)$, we consider the element $f^E $ in
$E\bigotimes_AA[V\bigotimes_A\dual{E}]$ given as
\eqbeg\label{fE}
f^E =\big( \sum_{k=1}^n e_k\otimes (t_{i_1}\otimes\dual{e}_k)\big)
\cdots \big(\sum_{k=1}^n e_k\otimes(t_{i_p}\otimes\dual{e}_k)\big),
\eqend
where $\dual{e}_1, \ldots, \dual{e}_n$ is the dual basis of
$\dual{E}$. The element $f^E=u(f)$, where $u$ is the universal map
$$ \xymatrix{ u \colon T_A(V) \ra
  E\bigotimes_AA[V\bigotimes_A\dual{E}].}$$
 Describing the correspondence given by $u$ for
monomials will suffice to describe the correspondence for arbitrary
elements. We have a unique decomposition 
$$f^E=\sum_{k=1}^n e_k\otimes f_k^E,$$
with $f_k^E \in A[V\bigotimes_A\dual{E}]$, for $k=1, \ldots, n$.
If we expand the defining expression \ref{fE} of $f^E$
  we get that 
$$ f^E =\sum_{\substack{1\leq k_i \leq n \\ i=1, \ldots, p}}
 e_{k_1}\cdots e_{k_p}\otimes (t_{i_1}\otimes\dual{e}_{k_1}) \cdots
(t_{i_p}\otimes\dual{e}_{k_p}).$$
Each monomial expression $e_{k_1}\cdots e_{k_p} $ in the free
$A$-module $E$, can be written
$\sum_{j=1}^nm^j(\underline{k})e_j$ for some $m^j(\underline{k}) \in
A$, with $j=1, \ldots, n$, and each ordered tipple $\underline{k}=k_1,
\ldots, k_p$. Therefore we get that
$$ f^E =\sum_{j=1}^n e_j\otimes \big(\sum_{\substack{1\leq k_i\leq n \\ i=1, \ldots ,p}} (t_{i_1}\otimes\dual{e}_{k_1}) \cdots
(t_{i_p}\otimes\dual{e}_{k_p}) \cdot m^j(\underline{k})\big).$$
In particular we have that
$$f^E_j =\sum_{\substack{1\leq k_i\leq n \\ i=1, \ldots ,p}} (t_{i_1}\otimes\dual{e}_{k_1}) \cdots
(t_{i_p}\otimes\dual{e}_{k_p}) \cdot m^j(\underline{k}),
$$
for each $j=1, \ldots, n$. Thus if we have a two-sided ideal in
$T_A(V)$ generated by an element $f$, then $A[V\bigotimes \dual{E}]/(f_1^E,
\ldots, f_n^E)$ represents $\NCWeil{A}{T_A(V)/(f)}{E}$.
\end{ex}

\section{Module restrictions}

In this section we will introduce the module restriction, which is the
main novelty of the article. From now on, the algebras
$A\ra B \ra R$ are all assumed to be commutative.

\subsection{Module structure} Recall that an $A$-module structure on an Abelian group $M$
is to have a ring homomorphism $\rho \colon A \ra
\operatorname{End}_{\bf Z}(M)$. The image of $\rho$ will factorise via
the subring of $A$-linear endomorphisms $\operatorname{End}_A(M)$,
making $\operatorname{End}_A(M)$ an $A$-algebra: The ring
$\operatorname{End}_A(M)$ is unital and associative, and the image of
the map $\mathrm{can} \colon A \ra \operatorname{End}_A(M)$ lies  in the centre.

\subsection{Extension of module structures} Let $g\colon A \ra B$ be a
homomorphism of rings. If $M$ is an $A$-module, then a $B$-module structure on the set $M$,
extending the fixed $A$-module structure, is a $B$-module structure on
$M$ that is compatible with the $A$-module structure. That is, a $B$-module structure on $M$ extending the
$A$-module structure is a ring homomorphism $\mu \colon B\ra
\operatorname{End}_A(M)$ making the commutative diagram
$$\xymatrix{ B  \ar[r]^-{\mu}   &\operatorname{End}_A(M) \\ 
 A\ar[u]^g \ar[ur]^{\operatorname{can}} & }.$$

\begin{rem} If we have an $A$-algebra homomorphism $\mu \colon B \ra
  \operatorname{End}_A(M)$, then the map will factorise via
  $\operatorname{End}_B(M)$. In particular we have that
  $\operatorname{End}_B(M)$ is a $B$-algebra, but in general
  $\operatorname{End}_A(M)$ is not a $B$-algebra.
\end{rem}

\begin{defn} Let $g\colon A\ra B$ and $f\colon B \ra R$ be commutative
  algebras and homomorphisms, and let
   $M$ be a $B$-module. We
define the functor $\Weilmodule{B}{R}{M}$, from the category of
$A$-algebras to sets,  by assigning to each
$A$-algebra $A'$ the set
$$ \Weilmodule{B}{R}{M}(A') =\left\{ \begin{array}{l} R\text{-module structures
on $M\bigotimes_AA'$,  extending}\\ \text{the fixed $B$-module
structure on $M\bigotimes_AA'$.} \end{array} \right\} $$
We call this functor the module restriction.
\end{defn}

\begin{thm}\label{Weilmodule is rep} Let $g\colon A \ra B$ and
  $f\colon B \ra R$ be homomorphism of commutative rings, and let $M$
  be a $B$-module. Assume that $M$, considered as an $A$-module, is
  projective and finitely generated. Then the functor
  $\Weilmodule{B}{R}{M}$ is naturally identified with $\NCWeil{B}{R}{E}$, where
  $E=\operatorname{End}_A(M)$, and in particular the functor $\Weilmodule{B}{R}{M}$ is
  representable. 
\end{thm}
\begin{proof}  Since $M$ is a $B$-module, we have that $M$ is also an
$A$-module. In particular we have that the $B$-module structure on $M$
extends the $A$-module structure. Thus, the $B$-module structure on
$M$ is given by an $A$-algebra homomorphism 
$ \mu \colon B \ra \operatorname{End}_A(M)$.
Let $A'$ be an $A$-algebra, and let $\xi$ be an
$A'$-valued point of the  module restriction
$\Weilmodule{B}{R}{M}$. Then we have that the $A'$-valued point
$\xi$ is a $A'$-algebra homomorphism making the following commutative diagram 
\begin{equation}\label{Thecommutativediagram}
\xymatrix{ R\bigotimes_AA' \ar[r]^-{\xi} &
  \operatorname{End}_{A'}(M\bigotimes_AA') \\
B\bigotimes_AA' \ar[u]^{f\otimes \id} \ar[r]^-{\mu\otimes \id}
&\operatorname{End}_A(M)\bigotimes_AA' \ar[u]_{\nu}, }
\end{equation}
where $\nu$ is the canonical map. Since $M$ is finitely generated and projective, the map
$\nu$ is an isomorphism
(\cite[4.3. Proposition 7]{Bourbaki_Alg1_3}).  The commutativity then
means that the $A'$-algebra homomorphism $\xi$ is $B$-linear. In other words, we  have a natural identification of functors
$$\Weilmodule{B}{R}{M} =\NCWeil{B}{R}{E},$$
with $E=\operatorname{End}_A(M)$. As $M$ is projective and finitely
generated, so is $E$, and the statement about representability
follows from Corollary \ref{Weil1stversion}.
\end{proof}

\subsection{A result about gluing} For any element $x$ in a ring $R$ we let
  $R_x=R[T]/(Tx-1)$ denote the localisation of $R$ at $x$.

\begin{lemma}\label{glue} Let $E$ be an $A$-algebra which is finitely
  generated and projective as an $A$-module, and let $A\ra B \ra R$
  be homomorphism of commutative rings. For any element $x$ in $R$ the scheme
  $\Spec(\NCWeil{B}{R_x}{E})$ is an open subscheme of
  $\Spec(\NCWeil{B}{R}{E})$. Moreover, if $y$ is another element of
  $R$, then we have 
$$ \Spec(\NCWeil{B}{R_{xy}}{E})=\Spec(\NCWeil{B}{R_x}{E})\cap
\Spec(\NCWeil{B}{R_y}{E}).$$
\end{lemma}

\begin{proof} Let $H$ denote the $A$-algebra that represents
  $\NCWeil{B}{R}{E}$, and let $\xi \colon R \ra E\bigotimes_AH$ denote the
  universal map. Let $x$ be an element of $R$. Then the universal map
  $\xi \colon R \ra E\bigotimes_AH$ factorises via $R\ra R_x$ if and
  only if $\xi(x)$ is a unit in $E\bigotimes_AH$. The element $\xi(x)$
  in the finitely generated and projective $H$-module $E\bigotimes_AH$ is a unit if and only if $d(x) =\det (e\mapsto e\cdot \xi(x))$ in
  $H$ is invertible. It follows that $H_{d(x)}$ represents
  $\NCWeil{B}{R_x}{E}$.
\end{proof}

\section{Module restriction in a geometric context}

In this section we will set our result about module restrictions in a
geometric context. 

\subsection{Relative rank and finite support} Let $\calE$ be
quasi-coherent sheaf of modules on an algebraic space $X$. The support
of $\calE$ is the closed subspace $\mathrm{Supp}(\calE)$ of $X$
determined by the
annihilator $\mathrm{ann}(\calE)$.


\begin{defn} Let  $g\colon X\ra S$ be a morphism of algebraic spaces, and $\calE$
  a quasi-coherent sheaf on $X$. We say that $\calE$ is {\em finite, flat
    of relative rank } $n$ over $S$, if $\calE$ is flat over $S$, and $\mathrm{Supp}(\calE)$ is finite over $S$, and the locally free
  ${\calO}_S$-module $g_*{\calE}$ has constant rank $n$. 
\end{defn}

\subsection{} If the quasi-coherent ${\calO}_X$-module $\calE$ is
finite, flat of rank $n$ over the base $S$, then it follows that
${\calE}$ is coherent ${\calO}_X$-module. In particular the underlying
set $|\mathrm{Supp}(\calE)| $ of the support, is precisely the
set of points where the sheaf ${\calE}$ is non-zero.
 
\subsection{} We denote by
${\mathscr C}oh_{X/S}^n$ the stack (\cite{LaumonMB} and \cite{Lieblich})  of quasi-coherent sheaves on
$X$, that are finite, flat and of relative rank $n$ over $S$.  A
morphism between two objects
$\calE$ and $\calF$ in ${\mathscr C}oh_{X/S}(T)$,
where $T\ra S$ is a scheme over $S$, 
 is an ${\calO}_{X\times_ST}$-module isomorphism $\varphi \colon
{\calE}\ra {\calF}$. If $T\ra S$ is a morphism  then
we have an isomorphism of stacks
\eqbeg\label{base change for stacks}{\mathscr C}oh^n_{X\times_ST/T} \simeq
{\mathscr C}oh^n_{X/S}\times_ST.
\eqend

\begin{lemma}\label{push forward map} Let $X\ra S$ be a separated map of algebraic spaces, and
  let $f\colon Y \ra X$ be a morphism of $S$-spaces. Then the
  push-forward gives a map $f_*\colon {\mathscr C}oh_{Y/S}^n \ra
  {\mathscr C}oh_{X/S}^n$. 
\end{lemma}

\begin{proof} Let $T\ra S$ be a morphism with $T$ a scheme, and let
  $\calE$ be an element of
  ${\mathscr C}oh^n_{Y\times_ST/T}$. Since the support
  $Z=\operatorname{Supp}(\calE)$ by assumption is finite over $T$, and
  the map $f\colon Y \ra X$ is separated, we have that the composition
  $f_T \colon Z\subseteq Y\times_ST \ra X\times_ST$ is finite
  \cite[Proposition 6.15]{EGAII}. In particular we have that
  ${f_{T}}_*\calE=\calE_X$ is quasi-coherent. The support of $\calE_X$ is the
  image of the support of $\calE$. Thus the support $\calE_X$ is
  proper and quasi-finite over $T$, hence finite. If $g_T \colon
  X\times_ST \ra T$ is the projection map, we have by definition that
  $g_T\circ f_T$ is the projection map from $Y\times_ST$. It follows
  that $f_{T*}{\calE}$ is finite, flat of relative rank $n$ over $T$. 
\end{proof}  

\begin{rem} Since the support of an element $\calE $ in ${\mathscr
C}oh^n_{X/S}(T)$ is finite, and in particular proper over the base $T$, it
follows that 
$$ {\mathscr C}oh^n_{U/S} \subseteq {\mathscr C}oh^n_{X/S}$$
is an open substack for an open subspace $U\subseteq X$ (see e.g. \cite{FGAIV}).
\end{rem}

\begin{thm}\label{affine push forward} Let $f\colon Y \ra X$ and
  $g\colon X \ra S$ be  morphisms of affine schemes $X,Y$ and
  $S$. Then we have that the push-forward map $f_*\colon
  {\mathscr C}oh^n_{Y/S} \ra {\mathscr C}oh^n_{X/S}$ is
schematically  representable.
\end{thm}

\begin{proof} We want to see that for arbitrary scheme $T$, the fibre product 
\eqbeg \label{2fiber}
T\times_{{\mathscr C}oh_{X/S}^n}{\mathscr C}oh^n_{Y/S}
\eqend
is representable by a module restriction. By Lemma \ref{glue} the module restrictions are Zariski
sheaves, and we may assume that $T$ is affine. We may, by \ref{base
  change for stacks}, assume  that $T=S$.  Let $\calE$ in
${\mathscr C}oh_{X/S}^n(S)$ be the element corresponding to
a given map $S\ra {\mathscr C}oh_{X/S}^n$. Let $S=\Spec(A)$, $X=\Spec(B)$, and let $M$ be the $B$-module
corresponding to the sheaf $\calE$ on $X$. Then $M$ is
projective and finitely generated as an $A$-module, and we have by Theorem \ref{Weilmodule is rep} the
$A$-algebra $\Weilmodule{B}{R}{M}$ where $Y=\Spec(R)$. We have a
natural map
$$ \xymatrix{\alpha \colon \Spec(\Weilmodule{B}{R}{M}) \ar[r] & 
S\times_{{\mathscr C}oh^n_{X/S}}{\mathscr C}oh^n_{Y/S}}$$
given as follows. Let $u'\colon \Spec(A') \ra
\Spec(\Weilmodule{B}{R}{M})$ be a morphism of affine schemes over
$S$. By the defining properties of the module restriction $\Weilmodule{B}{R}{M}$, the morphism $u'$
corresponds to a $R'=R\bigotimes_AA'$-module structure on
$M'=M\bigotimes_AA'$, extending the fixed $B'=B\bigotimes_AA'$-module
structure. Let $\xi \colon R' \ra
\operatorname{End}_{B'}(M')$ be the $B$-algebra homomorphism
corresponding to the $R'$-module structure on $M'$. And let
$\calF_{Y'}$ denote the corresponding quasi-coherent sheaf on
$Y'=Y\times_SS'$. Then $\alpha(u')=(s', \calF_{Y'}, \id)$, where $s'\colon S'\ra S$
is the structure map. The map $\alpha$ is a monomorphism, and we need
to see that it also is essentially surjective. 

Let $(s',\calF, \psi)$
be a $S'=\Spec(A')$-valued point of the fibre product \ref{2fiber}. Let
$N$ be the $R'=R\bigotimes_AA'$-module corresponding to the sheaf ${\calF}$ on
$Y\times_SS'$. Then $\psi$ corresponds to a $B'=B\bigotimes_AA'$-module isomorphism
$\psi \colon M'=M\bigotimes_AA' \ra N$. We get an induced $B'$-algebra
isomorphism
$$ \xymatrix{\tilde{\psi} \colon \operatorname{End}_{B'}(M') \ar[r] & 
\operatorname{End}_{B'}(N)}.$$
Finally, let $\xi' \colon R' \ra \operatorname{End}_{B'}(N)$ be the
$B'$-algebra homomorphism corresponding to the $R'$-module structure
on $N$.  The composition of
$\xi'$ with $\tilde{\psi}^{-1}$ gives a $B'$-algebra homomorphism $\xi
\colon R' \ra \operatorname{End}_{B'}(M')$. By the defining properties
of the module restriction $\Weilmodule{B}{R}{M}$ there exists a unique
$u'\colon S' \ra \Spec(\Weilmodule{B}{R}{M})$ corresponding to
$\xi$. Thus $\alpha(u')$ is isomorphic to $(s',\calF, \psi)$, and we
have shown that $\alpha$ is essentially surjective.
\end{proof} 
\begin{cor}\label{relatively representable} Let $f\colon X \ra S$ be a
  separated map of an algebraic spaces, and let  $f\colon Y \ra X$ be an
  affine morphism of $S$-spaces. Then the push-forward
  map
$$ \xymatrix{ f_{*}\colon {\mathscr C}oh^n_{Y/S} \ar[r] & 
{\mathscr C}oh^n_{X/S}}$$
is schematically representable. 
\end{cor}
\begin{proof} By \ref{base change for stacks} it suffices to show the result for affine base
  scheme $S$.  By Lemma \ref{glue} it suffices to show
  representability for fibres on $T$-valued points of ${\mathscr
    C}oh^n_{X/S}$, with affine $T$. Then the result follows from the
  theorem.
\end{proof}

\begin{lemma}\label{closed diagonal} Let $g\colon X\ra S$ be a separated map of algebraic spaces, and
let $f\colon Y \ra X$ be an affine morphism. Then the natural morphism
$$ {\mathscr C}oh^n_{Y/S} \times_{{ {\mathscr C}oh^n_{X/S}}}{\mathscr
  C}oh^n_{Y/S} \ra {\mathscr C}oh^n_{Y/S} \times_S {\mathscr
  C}oh^n_{Y/S}$$
is a closed immersion.
\end{lemma}

\begin{proof} We may assume that $S$ is affine. By \ref{base change for stacks} it suffices to check
  closedness for $S$-valued points. Let ${\calE}_1$ and ${\calE}_2$ we
  two $S$-valued points of ${\mathscr C}oh^n_{Y/S}$ such that
  $g_*f_*{\calE}_1=g_*f_*{\calE}_2$ are equal as $S$-modules.  Let
  $S=\Spec(A)$, and let $M$ be an $A$-module such that
  $\tilde{M}=g_*f_*{\calE}_1=g_*f_*{\calE}_2$. Let $\Spec(B)$ be an open
  subscheme of $X$, and let $f^{-1}(\Spec(B))=\Spec(R)$. The sheaves
  $\calE$ and $\calE'$ restricted to $\Spec(R)$ are given by two
  $R$-module structures on $M$, say $\xi_i \colon R \ra
  \operatorname{End}_A(M)$, for $i=1,2$. These two morphisms composed
  with the structure map $\varphi\colon B \ra R$, gives two
  $A$-algebra homomorphisms $\xi_i\circ \varphi\colon  B \ra
  \operatorname{End}_A(M)$. By Corollary \ref{closed conditions}
  the equality of these two maps is a closed condition on $A$.
   \end{proof}

\section{Weil restrictions revisited}
We will in this section define an open subfunctor of the module restriction that inherits properties as \'etaleness. 

\begin{prop}\label{gens about Weil} Let $f\colon B \ra R$ be homomorphism of commutative
  $A$-algebras, and let $\Weil{B}{R}=\NCWeil{B}{R}{B}$ denote the Weil
  restriction functor. Then we have that 
\begin{enumerate}
\item If $f\colon B \ra R$ is of finite presentation, then 
  $\Weil{B}{R}$ is of finite presentation.
\item If $f\colon B \ra R$ satisfies the infinitesimal lifting
  property for \'etaleness (respectively smoothness), then the functor 
$ \Weil{B}{R}$ satisfies the corresponding infinitesimal lifting
property.
\end{enumerate}
\end{prop}

\begin{proof} We will use the functorial characterisation \cite[8.14.2.2]{EGA4III} to prove the
  first assertion.  If we have a
directed system of $A$-algebras (and $A$-algebra homomorphisms)
$\{A_{\alpha}\}_{\alpha \in \calA}$, then we obtain an induced
injective map
\begin{equation}\label{bijection} \lim_{\to} \Weil{B}{R}(A_{\alpha}) \ra
\Weil{B}{R}(\lim_{\to}(A_{\alpha})).
\end{equation}
We need to see that this map \ref{bijection} is a bijection. Let
$\lim_{\to}A_{\alpha}=A'$, and let $\xi' \in \Weil{B}{R}(A')$. Then we obtain the following
  commutative diagram
$$ \xymatrix{ R \ar[dr]^{\xi'}&  \\
B \ar[u]^f \ar[r] & B\bigotimes_AA',}$$
and in particular we have that $\xi'$ is a $B$-algebra homomorphism. As $\lim_{\to}
(B\bigotimes_AA_{\alpha})=B\bigotimes_AA'$, and as $f \colon B \ra R$
is of finite presentation, we have a bijection 
$$ \lim_{\to }\xymatrix{(\mathrm{Hom}_{B\text{-alg}}(R,
    B\bigotimes_AA_{\alpha})) \ra \mathrm{Hom}_{B\text{-alg}}(R,
    B\bigotimes_AA') }.$$
Consequently $\xi'=\{\xi_{\alpha}\}$ is a sequence of $B$-algebra homomorphisms $\xi_{\alpha} \in \Weil{B}{R}(A_{\alpha})$. Thus the map
\ref{bijection} is a bijection, and $\Weil{B}{R}$ is of finite
presentation.
To check the two remaining assertions, let $A'$ be an $A$-algebra, and
$N\subseteq A'$ a nilpotent ideal. Let $\xi_N$ be an $A'/N$-valued
point of $\Weil{B}{R}$. We then obtain the following commutative
diagram
$$ \xymatrix{ R \ar[r]^{\xi_N} &B\bigotimes_AA'/N \\
B\ar[u]^f \ar[r] & B\bigotimes_AA'.\ar[u] }$$
Now, as $N\subseteq A'$ is nilpotent, the kernel of the canonical map 
$$\xymatrix{B\bigotimes_AA' \ra B\bigotimes_AA'/N=B/NB}$$
is nilpotent. Then if $f\colon B \ra R$ has an infinitesimal lifting
property, we obtain a lifting $\xi' \colon R \ra B\bigotimes_AA'$ of
$\xi_N$. Then $\xi'$ is an $A'$-valued point of $\Weil{B}{R}$, and we
have proved the last assertions. 
\end{proof} 

\begin{rem} It is well-known that the Weil restriction
$\Weil{B}{R}$ inherits properties as \'etaleness, smoothness, if $B\ra
R$ is \'etale, respectively smooth. In \cite{BLR}, e.g., these and
other properties are shown for the Weil restriction, however with some
assumptions that does not apply in our context. 
\end{rem}

\begin{ex} We give here an example showing that if $B\ra R$ is
  \'etale, then $\NCWeil{B}{R}{E}$ will not necessarily satisfy the
  infinitesimal lifting property. That fact was pointed out to us by Dan Laksov, who thereby corrected an error we had in a earlier version of this article. Consider first
  the matrices
$$ x=\begin{bmatrix} 0 & \epsilon \\ 0 & 0 \end{bmatrix} \quad \text{
  and } \quad  y =\begin{bmatrix} a_1+a_2\epsilon & b_1+b_2\epsilon \\ c_1+c_2\epsilon & d_1+d_2\epsilon \end{bmatrix},$$  
where the entries of the matrices are in some ring $A$,  where $\epsilon $ is a non-zero element  such that $\epsilon
^2=0$.  Since
$$ xy=\begin{bmatrix} a_1\epsilon & b_1\epsilon \\ 0  & 0\end{bmatrix}
\quad \text{ and } \quad yx =\begin{bmatrix} 0 & b_1\epsilon \\ 0 &
  c_1\epsilon \end{bmatrix},$$
these matrices do not in general  commute. However, when we set
$\epsilon =0$, the matrix $x$ becomes the zero matrix and the reduced
matrices clearly commute. Therefore we have the following. Let
$A=k[\epsilon]/(\epsilon^2)$, over a field $k$. Let $B=A[X]$ denote
the polynomial ring in the variable $X$ over $A$, and let $M=A\bigoplus A$. The matrix $x$
gives a $B$-module structure on $M$ by sending the variable $X$ to the
matrix $x$. Thus we have an $A$-algebra homomorphism $\mu \colon B \ra
\operatorname{End}_A(A\bigoplus A)=E$. We let $R=A[X,Y]/(Y^2-1)$, which
is \'etale over $B$ when the characteristic of $k$ is different from two.  We let furthermore $A'=A$, and the nilpotent ideal
$N=(\epsilon)\subseteq A$. We then have the following commutative diagram
$$ \xymatrix{ A \ar[r] & B=A[X] \ar[r]^f \ar[d]^{\mu} & R=A[X,Y]/(Y^2-1) \ar[d]^{\xi} \\
 & \operatorname{End}_A(A\bigoplus A) \ar[r] &
 \operatorname{End}_k(k\bigoplus k)=E\bigotimes_A k,}
$$
where $\xi \colon R \ra \operatorname{End}_k(k\bigoplus k)$ is determined
by sending $X$ to $0$ and sending $Y$ to the endomorphism given by the matrix
$\overline{y}=\begin{bmatrix} 0 & -1 \\ -1 & 0\end{bmatrix}$. 
Any lifting of $\overline{y}$ to an element in $\operatorname{End}_A(A\bigoplus A)$ is
of the form 
$$ y=\begin{bmatrix} a_2\epsilon &-1+ b_2\epsilon \\-1+ c_2\epsilon &
  d_2\epsilon \end{bmatrix}$$
with elements $a_2, b_2, c_2, d_2$ in $k$. From the considerations
above we have that no such lifting will commute with the matrix
$x$. Therefore there exist no $B$-algebra homomorphism $\tilde{\xi}
\colon R \ra \operatorname{End}_A(A\bigoplus A)$ that extends
$\xi$. Thus, even if $f\colon B \ra R$ is \'etale, the $A$-algebra
$\NCWeil{B}{R}{E}$, and
$\Weilmodule{B}{R}{M}$, are not necessarily \'etale.
\end{ex}

\begin{ex} Our next example shows  that even if $B\ra R$ is of
  finite presentation, the $A$-algebra representing $\NCWeil{B}{R}{E}$
  is not of finite presentation. It follows, though, from the
  constructions that if $B\ra R$ is of finite type, then the
  $A$-algebra $\NCWeil{B}{R}{E}$ is of finite type. 

Let
  $A=k[w_i]_{i\geq 0}$ be the polynomial ring in a countable
  number of variables $w_1, w_2, \ldots $ over some ring $k$. Let
  $M=A\bigoplus A$ be the free $A$-module of rank 2. From the polynomial
  ring in one variable $A[T]$ over $A$,  we obtain an $A[T]$-module structure on
  $M$ by sending $T$ to the matrix 
$$ t=\begin{bmatrix} 0 & 1 \\ 1 & 0 \end{bmatrix}.$$
Let $B=A[X_i]_{i\geq 1}$ be the polynomial ring in the variables $X_1,
X_2, \ldots, $ over $A$. For each $i$ we consider the matrix
$$ x_i =\begin{bmatrix} w_i & 0 \\ 0 & w_{i+1} \end{bmatrix}.$$
Since the diagonal matrices commute, we get an $B$-module structure on
$M$ by sending the variable $X_i$ to the matrix $x_i$. One checks that
the two
$A$-algebra homomorphisms $\mu \colon B \ra E=\operatorname{End}_A(M)$
and $u \colon A[T] \ra E$ commute if and only if $w_i=w_{i+1}$, for all $i=1, 2,
\ldots $. Thus, the closed condition on $A$ over where the two maps
$\mu$ and $u$ commute is given by $A/(w_{i}-w_{i+1})_{i\geq 1}$, which
is not of finite presentation. As finite presentation is preserved
under specialisation, we get  that the  $A$-algebra
$\NCWeil{B}{B\bigotimes_AA[T]}{E}$ can  not be of finite presentation
either. 
\end{ex}

\subsection{Isomorphic image functor} Let $A\ra B\ra R$ be
homomorphisms of commutative rings, and let $\mu \colon B \ra E$ be a
homomorphism of
$A$-algebras, where $E$ is not necessarily commutative. We will define a subfunctor
\begin{equation}\label{defn iso image}
\Iso{B}{R}{E} \subseteq \NCWeil{B}{R}{E}.
\end{equation}
Recall that an $A'$-valued point of $\NCWeil{B}{R}{E}$ is an
$A'$-algebra homomorphism $\xi'$ that fits into the
following commutative diagram
\begin{equation}\label{A'valued point}
\xymatrix{ R\bigotimes_AA' \ar[dr] ^{\xi'} & \\ B\bigotimes_AA' \ar[u]
  \ar[r]_{\mu \otimes 1} & E\bigotimes_AA'.}
\end{equation}
In particular we have an induced map $B\bigotimes_AA' \ra \im(\xi')$,
where $\im(\xi')$ denotes the image of the homomorphism $\xi'$. For any $A$-algebra $A'$, we let
$$
\Iso{B}{R}{E}(A')=\{ \xi' \in \NCWeil{B}{R}{E} \text{ such that } \im (\mu \otimes 1)=\im(\xi') \}.
$$

\begin{prop}\label{Isorepresentability} Let $A\ra B \ra R$ be homomorphism of commutative rings,
  and let $\mu \colon B \ra E$ be an $A$-algebra homomorphism with $E$ not necessarily commutative. Assume that $E$ is finitely
  generated and projective as an $A$-module, and that $B\ra R$ is of
  finite type. Then the functor $\Iso{B}{R}{E}$ is representable by an open
  subscheme of $\NCWeil{B}{R}{E}$.
\end{prop}

\begin{proof} Let $H$ be the $A$-algebra representing the functor
  $\NCWeil{B}{R}{E}$, Corollary \ref{NCWeil is rep}. Let $\xi
  \colon R\bigotimes_AH \ra E\bigotimes_AH$ denote the
  universal element of $\NCWeil{B}{R}{E}$,  and let $\im( \xi)  \subseteq E\bigotimes_AH$ denote the
  image of $\xi$. Let furthermore, $B_H$ denote the image of
  $\mu\otimes 1\colon B\bigotimes_AH \ra E\bigotimes_AH$.  We have the inclusion of
  $B_H$-modules $B_H \subseteq \im(\xi) \subseteq E\bigotimes_AH$. Any element $x\in E\bigotimes_AH$
  gives by multiplication, an $H$-linear endomorphism on $E\bigotimes_AH$. The
  Cayley-Hamilton theorem guarantees that the element $x$ will satisfy
  its characteristic polynomial, and consequently that $x$ is integral over
  $H$. From this we deduce the following two consequences. Firstly, since $R$ is
  finite type over $B$ we have that $\im(\xi)$ is a finitely generated
  $B_H$-module. In particular the quotient $B_H$-module $\im(\xi)/B_H$ has
closed  support $Z\subseteq \Spec(B_H)$ given by the annihilator ideal
  $\operatorname{ann}_{B_H}(\im(\xi)/B_H)$. Secondly, as $B_H$ is integral over $H$,
  the corresponding morphism $g \colon \Spec(B_H) \ra \Spec(H)$ is
  closed. Thus $g(Z)$ is the closed subscheme given by the ideal
  $\operatorname{ann}_{B_H}(\im(\xi)/B_H)\cap H$. And as $g^{-1}(g(Z))=Z$ it
  is clear that the open subscheme $U=\Spec(H) \setminus g(Z)$
  represents $\Iso{B}{R}{E}$.
\end{proof}

\begin{rem} Similar result can be found in  \cite{Gamma}. 
\end{rem}

\subsection{Properties of the Isomorphic image functor} We keep the
notation introduced above, and assume that the $A$-algebra $E$ is
finitely generated and projective as an $A$-module.  Let $H$ be the $A$-algebra representing
$\NCWeil{B}{R}{E}$, and let $\xi \colon R\bigotimes_AH \ra
E\bigotimes_AH$ denote the universal element. We then have the
following commutative diagram
\begin{equation}\label{universal diagram}
\xymatrix{ R\bigotimes_AH \ar[r] & R_H \ar[dr] & & \\
B\bigotimes_AH \ar[u] \ar[r] & B_H \ar[u] \ar[r]^i & \im(\xi)
\ar[r] &  E\bigotimes_AH, }
\end{equation}
where $B_H$ is the image of $\mu \otimes 1\colon
B\bigotimes_AH \ra E\bigotimes_AH$, and where
$R_H=R\bigotimes_BB_H$.

\subsubsection{}  When we restrict the diagram \ref{universal diagram} to the open
subscheme $U\subseteq \Spec(H)$ representing $\Iso{B}{R}{E}$,
Proposition \ref{Isorepresentability}, we get
by definition that the map of ${\calO}_U$-modules
\begin{equation}\label{restriction to U}
i_{|U} \colon {B_H}_{|U} \ra \im(\xi)_{|U}
\end{equation}
is surjective. As $U\subseteq \Spec(H)$ is an open immersion, and in
particular a flat map that preserves injectivity, we get that the
restriction morphism \ref{restriction to U} is an isomorphism.

\begin{prop}\label{Iso = Weil} Let $A\ra B \ra R$ be homomorphism of commutative rings,
  and let $\mu \colon B \ra E$ be an $A$-algebra homomorphism,  with $E$
  not necessarily commutative. Assume that $E$ is finitely
  generated and projective as an $A$-module, and that $B\ra R$ is of
  finite type. Let $H$ be the $A$-algebra representing
  $\NCWeil{B}{R}{E}$. Let $B_H$ denote the image of the
  composite map $\mu \otimes 1 \colon B\bigotimes_AH \ra E\bigotimes_AH$, and let
  $R_H=R\bigotimes_BB_H$. Then we have that the
  functor $\Iso{B}{R}{E}$ equals the Weil restriction
  $\Weil{B_H}{R_H}$. In particular we have that the Weil restriction $\Weil{B_H}{R_H}$ is representable by a scheme.
\end{prop}

\begin{proof} Let $A'$ be an $A$-algebra,  and let
  $\varphi \colon \Spec(A') \ra
  \Spec(H)$ be a morphism that factorises via the open immersion
  $U\subseteq \Spec(H)$, where $U$ represents $\Iso{B}{R}{E}$. From
  the diagram \ref{universal diagram} we obtain the commutative
  diagram 
\begin{equation}
\xymatrix{ R_H\bigotimes_AA' \ar[dr] & \\ B_H\bigotimes_AA' \ar[u]
  \ar[r]^{i\otimes 1} & \im(\xi) \bigotimes_AA'.}
\end{equation}

As $\varphi $ factors via $U\subseteq \Spec(H)$ we get from \ref{restriction to U}, that $i\otimes 1$ in the diagram above is an
isomorphism. We have that
$B_H\bigotimes_HA'=B_H\bigotimes_AH\bigotimes_HA'=B_H\bigotimes_AA'$. Consequently
the composition 
$$\xymatrix{ R_H\bigotimes_HA'=R_H\bigotimes_AA'
  \ar[r] & 
  \im(\xi)\bigotimes_AA' \ar[r]^{j} &
    B_H\bigotimes_AA'}$$
is an $A'$-valued point of the Weil restriction
$\Weil{B_H}{R_H}$, with $j=(i\otimes 1)^{-1}$.

Conversely, let  $s\colon R_H\bigotimes _AA' \ra
B_H\bigotimes_AA'$ be an $A'$-valued point of the Weil
restriction. We then get the commutative diagram 
$$ \xymatrix{R\ar[r] & R_H\bigotimes_AA' \ar@<1ex>[d]^s & 
\\  B\ar[u] \ar[r] & B_H\bigotimes_AA' \ar[u] \ar[r]^{i\otimes
    1} &\im(\xi)\bigotimes_AA'.}
$$ 
The composition $R\ra \im(\xi)\bigotimes_AA'\ra E\bigotimes_AA'$ is an
$A'$-valued point $\xi'$ of $\NCWeil{B}{R}{E}$. By construction the
image of $\xi'$ equals the image of $\mu\otimes 1 \colon
B\bigotimes_AA' \ra E\bigotimes_AA'$. Hence we
have that $\xi'$ is an $A'$-valued point of $\Iso{B}{R}{E}$. We have
now constructed a functorial bijection between the $A'$-valued points
of $\Iso{B}{R}{E}$ and the $A'$-valued points of the Weil restriction $\Weil{B_H}{R_H}$.
\end{proof}

\begin{rem} Note that the $A$-algebra $B_H$ is not assumed to be
  finitely generated or projective as an $A$-module.  
\end{rem}

\begin{cor}\label{etaleness of iso} Let $A\ra B \ra R$ be homomorphism of commutative rings,
  and let $\mu \colon B \ra E$ be an $A$-algebra homomorphism, with $E$
  not necessarily commutative. Assume that $E$ is finitely
  generated and projective as an $A$-module, and that $B\ra R$ is
  \'etale (smooth). Then the scheme representing the functor
  $\Iso{B}{R}{E}$ is \'etale (respectively smooth).
\end{cor}

\begin{proof} If $f\colon B\ra R$ is \'etale, or smooth, then in particular
  it is of finite type. Hence, by Proposition
  \ref{Isorepresentability}, the functor $\Iso{B}{R}{E}$ is
  representable by a scheme. Moreover, by Proposition \ref{Iso = Weil}
  we have that $\Iso{B}{R}{E}$ equals the Weil restriction
  $\Weil{B_H}{R_H}$, where we use the notation of Proposition \ref{Iso
    = Weil}. As $f\colon B \ra R$ is \'etale, or smooth, then we have
  by base change that $B_H \ra R_H$ is \'etale, or respectively
  smooth. The result then follows by Proposition \ref{gens about
    Weil}.
\end{proof}

\begin{cor}\label{openessofiso} Let $f\colon B\ra R$ be an $A$-algebra homomorphism. Let $M$ be
  an $B$-module, and let $E=\operatorname{End}_A(M)$. Assume that $M$
 is finitely generated and projective as an $A$-module.  Then we have that for any $A$-algebra $A'$, the set of $A'$-valued points of $\Iso{B}{R}{E}$
  corresponds to $R'=R\bigotimes_AA'$-module structures on
  $M\bigotimes_AA'$, extending the fixed $B'=B\bigotimes_AA'$-module structure,
  and such that the induced map of supports
$$\xymatrix{B'/\operatorname{ann}_{B'}(M\bigotimes_AA') \ra R'/\operatorname{ann}_{R'}(M\bigotimes_AA')}$$
is an isomorphism.
\end{cor}
\begin{proof} As $M$ projective and
  finitely generated $A$-module, we have that
  $\operatorname{End}_A(M)\bigotimes_AA'=\operatorname{End}_{A'}(M\bigotimes_AA')$,
  for any $A$-algebra $A'$. Then, for
  a given $A$-algebra $A'$, we have that the kernel of $\mu \otimes 1 \colon
  B\bigotimes_AA' \ra E\bigotimes_AA'$ is the annihilator ideal. Thus,
 $B'/\operatorname{ann}_{B'}(M\bigotimes_AA')$ is the image of $\mu
 \otimes 1$, with $B'=B\bigotimes_AA'$. If $\xi'$ is an $A'$-valued point of $\Iso{B}{R}{E}$, then we have by definition that
$$\xymatrix{B'/\operatorname{ann}_{B'}(M\bigotimes_AA') = \im(\mu\otimes 1)=\im (\xi')=R'/\operatorname{ann}_{R'}(M\bigotimes_AA'),}$$
where $R'=R\bigotimes_AA'$.
\end{proof}

\subsection{Situation with commuative rings} With $E=\operatorname{End}_A(M)$ we have the subfunctor  $\Iso{B}{R}{E}$ of the modul restriction $\NCWeil{B}{R}{E}$.  We end this section by looking at the special situation when $M$ is not only an $A$-module, but also an $A$-algebra. Let $g\colon A\ra B$ be a homomorphism
of commutative rings, and consider $B$ as an $A$-module via this
homomorphism. We then have the commutative diagram
$$ \xymatrix{B \ar[r]^c & \operatorname{End}_A(B) \\ A\ar[u]^g
  \ar[ur]_{\mu} & }$$
where $c\colon B\ra \operatorname{End}_A(B)$ is the canonical map. We
have furthermore that the canonical map $c \colon B \ra \operatorname{End}_A(B)$
identifies $B$ with its image.

\begin{lemma}\label{module restriction rings} Let $g\colon A \ra B$
  and $f\colon B \ra R$ be homomorphisms of commutative rings. We have that if $B\ra B'$ is an
  $A$-algebra homomorphism, and we let $E=\operatorname{End}_A(B')$,
  then we have equality of functors $\NCWeil{B}{R}{E}=\Iso{B}{R}{E}$.
\end{lemma} 
\begin{proof} Any $B$-linear
  algebra homomorphism $\xi \colon R \ra \operatorname{End}_A(B)$
  will factorise via the inclusion $B\subseteq
  \operatorname{End}_A(B)$. For any $x\in R$ we have that $\xi(x)=\xi_x$ is an
  endomorphism on $B$. We identify $\xi_x$ with its evaluation on the
  unit.
\end{proof}

\section{Representability of the Quot functor}

\subsection{Isomorphic support} Let $X\ra S$ be separated morphism of
algebraic spaces, and let $f\colon Y \ra X$ be a morphism of
$S$-spaces. For each scheme $T$, and for any element ${\calE}$ in
$\Coh{Y}(T)$ we have an induced map on supports
\begin{equation}\label{induced map of support}
 f_{T |} \colon \mathrm{Supp}({\calE}) \ra \mathrm{Supp}( f_*{\calE}).
\end{equation}
Let 
$$ {\stackU}_{Y\to X} \subset \Coh{Y}$$
denote the substack, whose objects are ${\calO}_{Y_T}$-modules
${\calE} \in \Coh{Y}(T)$ such that the induced map on supports is
an closed immersion.

\begin{prop}\label{open substack} Let $X\ra S$ be a separated map of algebraic spaces. Let $f\colon Y \ra X$ be a
  $S$-morphism that is affine and of finite type. Then induced map of stacks 
$$ {\stackU}_{Y\to X} \ra \Coh{Y}$$
is a representable open
immersion. If furthermore, $f\colon Y \ra X$ is \'etale (smooth), then the induced composite map
$$ {\stackU}_{Y\to X} \ra \Coh{Y} \ra \Coh{X}$$
is \'etale (respectively smooth).
\end{prop}

\begin{proof} The results follows 
  by Proposition \ref{Isorepresentability} and Corollary \ref{etaleness of iso}.
\end{proof}



\subsection{The Quot stack} Fix a quasi-coherent sheaf $F_X$ on an
algebraic space $X\ra S$. For any $S$-scheme $T$ we let $F_{X_T
}$ denote the pull-back of $F_X$ along the first projection
$p_X\colon X\times_ST \ra X$. The $T$-valued points of the quot stack
$\stackquot ^n_{F_X/S}$ are all ${\calO}_{X\times_ST}$-module morphisms
$q \colon F_{X_T} \ra \calE$, from $F_{X_T}$ to a
quasi-coherent sheaf $\calE$ on $X\times_S T$, where $\calE$ is flat, finite of relative rank $n$
over $T$. A morphism between two objects $q\colon F_{X_T}\ra \calE$
and $q'\colon F_{X_T}\ra \calE'$ is an ${\calO}_{X\times_ST}$-module
isomorphism $\varphi \colon \calE\ra \calE'$ such that $q'=\varphi
\circ q$. 

\begin{rem}Note that the maps $q\colon F_{X_T} \ra \calE$ are not assumed to
be surjective, and in particular the $T$-valued points of the quot
stack $\stackquot ^n_{F_X/S}$ are not quotients of $F_X$. The definition of the quot stack is motivated by the
definition of the Hilbert stack in \cite{RydhHilb},  and
in \cite{Olsson}.  
\end{rem}

\subsection{Identification of pull-backs} We will in the sequel of this article return to
a particular situation that we describe below. Let $T\ra S$ be a
morphism. Then we have the following Cartesian diagram
\eqbeg\label{notation cartesian}\xymatrix{ Y\times_ST \ar[r]^{f_T} \ar[d]^{p_Y} & X\times_ST
  \ar[d]^{p_X} \\
Y \ar[r]^f & X.}
\eqend
For any sheaf $F_X$ on $X$ there is a canonical identification between
the two sheaves
$f_T^*F_{X_T}=f_T^*p_X^*F_X$ and $p_Y^*f^*F_X$. We  will denote both these two 
sheaves with $F_{Y_T}$. 
We have, furthermore, a natural map
$$ \xymatrix{c \colon \stackquot^n_{F_X/S} \ar[r] & {\mathscr C}oh^n_{X/S}}$$
that takes a $T$-valued point of the
quotient stack $q\colon F_{X_T} \ra \calE$ to the sheaf
$\calE$. 

\begin{lemma}\label{cartesian} Let $X\ra S$ be a separated map of algebraic spaces, and
  let $f\colon Y \ra X$ be a morphism of $S$-spaces. For any
  quasi-coherent sheaf $F_X$ on $X$, we have the
  Cartesian diagram
$$\xymatrix{ \stackquot^n_{F_X/S} \ar[r]^c & {\mathscr C}oh_{X/S}^n \\
  \stackquot^n_{F_Y/S} \ar[u]^{f_*}\ar[r]^c &
  {\mathscr C}oh_{Y/S}^n \ar[u]^{f_{*}}. }
$$
\end{lemma}

\begin{proof} We first establish the map $f_*$  from $\stackquot ^n_{F_Y/S}$
  to $\stackquot ^n_{F_X/S}$. Let $q\colon F_{Y_T} \ra \calE$ be a $T$-valued
  point of $\stackquot ^n_{F_Y/S}$. The canonical map
  $F_{X_T} \ra f_{T*}f^{*}_T F_{X_T}$, where we use the notation of
  \ref{notation cartesian}, 
combined with the
  identification $f_T^*F_{X_T}=F_{Y_T}$, gives the composition
\begin{equation}\label{push forward}
\xymatrix{F_{X_T} \ar[r] &  f_{T*}f_T^*F_{X_T}=f_{T*}F_{Y_T} \ar[r] &
  f_T^*\calE}.
\end{equation}
By Lemma \ref{push forward map} the above sequence is a $T$-valued
point of $\stackquot ^n_{F_X/S}$. Now as we have have the map,  the
proof is  a formal consequence of adjunction.
\end{proof}

\subsection{The Quot functor} Let $X \ra S$ be a morphism of algebraic
spaces, and let $F_X$ be a quasi-coherent sheaf on $X$. The Quot functor 
$\Quot{X}{F}$ defined by Grothendieck (\cite{FGAIV}, \cite{Artinformalmoduli}) is the functor that to each $S$-scheme $T\ra S$ assigns
the set of surjective ${\calO}_{X_T}$-module maps $q \colon F_{X_T}
\ra \calE$, where $\calE$ is finite, flat and of relative rank
$n$. Two surjective maps $q\colon F_{X_T}\ra \calE$ and
$q'\colon F_{X_T}\ra \calE'$ are considered as equal if their
kernels coincide as subsheaves of $F_{X_T}$.

\subsection{} Let $q\colon F_{X_T}\ra \calE$ be a $T$-valued
point of $\Quot{X}{F}$. We define
$$ \iota (\calE) = {\calO}_{X\times_ST}/\mathrm{ker}q.$$
This determines a map  $\iota\colon \Quot{X}{F} \ra
\stackquot^n_{F_X,S}$.

\subsection{Subfunctors} We will identify two subfunctors of the Quot
functor, and then in the next lemma relate these subfunctors with the
stacks introduced earlier. Let $f\colon Y \ra X$ be a morphism of
$S$-spaces, with $X\ra S$ separated, and let $F_X$ be a quasi-coherent
sheaf on $X$. We want to consider the two following subfunctors 
$$\Omega^F_{Y\to X} \subseteq \omega^F_{Y\to X} \subseteq \Quot{Y}{F}.$$
We define $\omega^F_{Y\to X}$ as the subfunctor of
$\Quot{Y}{F}$ whose $T$-valued points are surjective
${\calO}_{Y_T}$-module maps $q\colon
F_{Y_T} \ra \calE$, where $\calE$ is finite, flat of rank $n$ over
$T$, such that the induced
map of ${\calO}_{X_T}$-modules
$$ \xymatrix{f_*(q) \colon F_{X_T} \ar[r] & f_{T*}\calE}$$ 
is surjective, where $f_*$ is the push forward map \ref{push forward}.
And we define $\Omega^F_{Y\to X}$ with the further requirement that the
induced map of supports (\ref{induced map of support}) is a closed immersion.

\begin{lemma}\label{open} Let $X\ra S$ be a separated map of algebraic
  spaces, and let $f\colon Y\ra X$ be a map of $S$-spaces. For any
  quasi-coherent sheaf $F_X$ on $X$ we have that $\omega_{Y\to X}^{F}$ is an open subfunctor
  of $\Quot{Y}{F}$. Moreover, we have the following Cartesian
  diagrams
$$\xymatrix{ \stackU _{Y\to X} \ar[r] &  \Coh{Y} \ar[r]^{f_*} &  \Coh{X}\\
 & \stackquot^n_{F_Y/S} \ar[r]^{f_*}\ar[u]^c &
  \stackquot^n_{F_X/S} \ar[u]^c \\
\Omega_{Y\to X}^F \ar[r] \ar[uu] & \omega_{Y\to X}^F \ar[r] \ar[u] & \Quot{X}{F}\ar[u]^{i}
}$$

In particular we have that
\begin{enumerate}
\item If the morphism $f\colon Y \ra X$ is  affine, then the morphism 
$\stackquot^n_{F_Y/S} \ra \stackquot^n_{F_X/S}$ is schematically
representable. 
\item If  $f\colon Y \ra X$ is affine and of finite
type, then $\Omega^F_{Y\to X}$ is an open subfunctor of
$\Quot{Y}{F}$, and  $\Omega^F_{Y\to X} \ra \Quot{X}{F}$ is schematically
representable.
\item Finally, if $f\colon Y \ra X$ is affine and \'etale,
then $$\Omega^F_{Y\to X} \ra \Quot{X}{F}$$
is a schematically representable, \'etale morphism.
\end{enumerate}
\end{lemma}

\begin{proof} We first show openness of $\omega_{Y\to X}^F$. Let $q\colon F_{Y_T} \ra  \calE$ be a $T$-valued point
  of $\omega^F_{Y\to X}$. By Lemma \ref{push forward map} the sheaf
  $f_{T*}\calE$ is quasi-coherent on $X\times_ST$. Furthermore, since $f_{T*}\calE$
  is finite over the base it follows that surjectivity of the map 
  $f_*(q)\colon F_{X_T} \ra f_{T*}\calE$ is an open condition on the
  base, proving the first claim. 

About the Cartesian diagrams. The upper right diagram is Cartesian by
Lemma \ref{cartesian}. We consider the lower right diagram.

We start by noting that there is a natural
map $\alpha \colon \omega^F_{Y\to X} \ra P$, where $P$ is the fibre
product in question. If $q\colon F_{Y_T} \ra \calE$ is a
$T$-valued point of $\omega^F_{Y\to X}$, then by assumption $f_*(q)
\colon F_{X_T} \ra f_{T*}\calE$ is surjective. So $\alpha(q)=(f_*(q),
i(q), \id)$. The map $\alpha$ is full and faithful.

Let $(q_X, s, \psi)$ be a $T$-valued point of the fiber product $P$,
where $q_X \colon F_{X_T} \ra \calE$ is surjective, $s\colon
F_{Y_T} \ra \calF$ is a map of ${\calO}_{Y\times_ST}$-modules, $\calE$ and $\calF$ are finite, flat
of relative rank $n$ over the base, and where $\psi $ is an
isomorphism of ${\calO}_{X\times_ST}$-modules, 
making the commutative diagram
$$ \xymatrix{ f_{T*}f_T^*F_{X_T}=f_{T*}F_{Y_T} \ar[r] & f_{T*}\calF \\
F_{X_T} \ar[r]^{q_X} \ar[u] & \calE \ar[u]^{\psi}. }
$$
Consider now the pull-back $f_T^*(q_X)$ composed with the adjoint of
$\psi$, 
$$\xymatrix{ q_Y \colon F_{Y_T}=f_T^*F_{X_T} \ar[r] &f^*_T\calE \ar[r]
  &\calF.}$$
We claim that the $\calO_{Y\times_ST}$-module map $q_Y$ is
surjective. To see this we may
restrict ourselves to the support $\operatorname{Supp}(\calF)\subseteq
Y\times_ST$ of $\calF$. The restriction of $f \colon
Y\times_ST \ra X\times_ST$ to the support of $\calF$ is finite
(\cite[Proposition 6.15]{EGAII}), and  in
particular affine. It is then clear that since $\psi \colon \calE \ra f_{T*}\calF$ is an
isomorphism, and then in  particular surjective, the adjoint map $f^*\calE
\ra \calF$ is also surjective. Since $q_X$ is surjective by
assumption, so is its pull-back $f_T^*(q)$. Thus, the map $q_Y$ is the
composition of two surjective maps, and therefore 
surjective. So $q_Y$ is consequently a $T$-valued point of
$\Quot{Y}{F}$. Moreover, since $\psi$ is an isomorphism it follows that $q_Y$ is
a $T$-valued point of $\omega^F_{Y\to X}$. We then have that
$f_*(q_Y)$ is isomorphic to $(q_X, s, \psi)$, and thus that $\alpha$
is essentially surjective. The leftmost diagram is proven to be
Cartesian in a similar way.
\end{proof}

\begin{thm} Let $X\ra S$ be a separated morphism of algebraic spaces, and let $F_X$ be a quasi-coherent sheaf on $X$. For each integer
  $n$, the functor $\Quot{X}{F}$ is representable by a separated algebraic
  space.
\end{thm}

\begin{proof} As the Quot functor commutes with base change, we may reduce to the case with the base $S$ being an
  affine scheme $S=\Spec (A)$. Moreover, any $T$-valued point $q
  \colon F_{X_T} \ra \calE$ of $\Quot{F}{X}$ is such that the support $\operatorname{Supp}(\calE) $
  is finite over the base. Hence  the support of the quotient $\calE$
  is contained in an open quasi-compact $U\subseteq
  X$. Therefore we have that
$$ \lim_{\substack{U\subseteq X \\ \text{open, q-compact}}}\Quot{U}{F} =\Quot{X}{F}.$$
Hence it suffices to show the theorem for $X\ra
  S$ being quasi-compact. With $X\ra S$ quasi-compact we can find an
  affine scheme $Y\ra S$ with an \'etale, affine and surjective map $f\colon Y\ra
  X$. With affine schemes $Y\ra S$  we have that  $\Quot{Y}{F}$ is represented by a
  scheme (\cite{affinequot}). By Lemma \ref{open} (2) we get that $\Omega_{Y\to
    X}^F$ is open in $\Quot{Y}{F}$, hence a
  scheme. Lemma \ref{open} (3) gives that the induced map
\begin{equation}
\label{cover} \xymatrix{\Omega_{Y\to X}^{F} \ar[r] & \Quot{X}{F}}
\end{equation}
is representable and  \'etale. 

We then want to see that the map (\ref{cover}) is surjective. Let $k$ be a field
and $F_{X_k}\ra \calE$ be a $\Spec(k)$-valued point of $\Quot{X}{F}$
where we use the notation $X_k=X\times_S\Spec(k)$.  We want to show
that there exists a separable field extension $k\ra L$ such that the
corresponding $\Spec(L)$-valued point lifts to $\Omega^{F}_{Y\to X}$. 

The reduced support
$Z=|\operatorname{Supp(\calE)}|$ is a  disjoint union  of a
finite set of points, given by finite field extensions $k \ra
k_i$ with $i=1, \ldots, m$. Then $f^{-1}(\Spec(k_i))$ is also a finite
union of points $\sqcup_{j_i=1}^{m_i}\Spec(L_{j_i})$, with $k_i \ra L_{j_i}$ a finite separable field
extension for $j_i=1, \ldots, m_i$. There exists a finite separable field extension $k \ra
L$ such that the induced map $k_i\bigotimes_kL \ra L_{j_i} \bigotimes_kL$
splits, for all $i=1, \ldots, m$ and all $j_i=1, \ldots, m_i$. Then
$$\xymatrix{f^{-1}(Z) \times_{\Spec(k)}\Spec(L) \ar[r] & Z\times_{\Spec (k)}\Spec
(L) }$$
has a section, and we have that the corresponding $\Spec(L)$-valued
point of
$\Quot{X}{F}$ lifts to $\Omega^{F}_{Y\to X}$. We then have proven
surjectivity, and consequently, accordingly to definition in
\cite{RG},  that $\Quot{X}{F}$ is an algebraic space. That the
algebraic space representing $\Quot{X}{F}$ is separated follows from
Lemma \ref{closed diagonal} and the Cartesian diagrams \ref{open}.
\end{proof}

\begin{rem} With $F_X={\calO}_X$ the structure sheaf on $X$, the Quot
  functor $\Quot{X}{F}$ is the Hilbert functor
  $\underline{\operatorname{Hilb}}^n_{X/S}$. The situation with the Hilbert scheme
  was considered in \cite{ES}, and a similar approach for Hilbert
  stacks was done in \cite{RydhHilb}.  Note that when $F_X={\calO}_X$ then we
  get by Lemma \ref{module restriction rings} that $\omega^F_{Y\to X}=\Omega^F_{Y\to X}$.
\end{rem}

\begin{rem} The separated assumption of $X\ra S$ is a necessary
condition for representability \cite{LS}.  On the other hand there exist examples of separated
schemes $X\ra S$ for which the Quot functor is not represented by a
scheme \cite{algsp}, but only an algebraic space. Thus when
considering representability, the setting with separated algebraic spaces
$X\ra S$ is the natural one.
\end{rem}

\begin{rem} The above result in its generality  is not covered by the
  result of Artin \cite{Artinformalmoduli}; we have no restriction on
  the base space $S$, and we do not assume that $X\ra S$ is of locally
  finite type.
\end{rem}
\bibliographystyle{dary}
\bibliography{paper}
\end{document}